\documentclass[12pt,a4paper]{article}
\usepackage[OT1]{fontenc}
\usepackage[latin1]{inputenc}
\usepackage[english]{babel}
\usepackage{amsfonts}
\usepackage{amsthm}
\newtheorem{thm}{Theorem}[section]
\newtheorem{prop}[thm]{Proposition}
\newtheorem{lem}[thm]{Lemma}
\newtheorem{cor}[thm]{Corollary}

\theoremstyle{definition}
\newtheorem{rem}[thm]{Remark}
\newtheorem{exmp}[thm]{Example}

\def\h{{\cal H}}
\def\fred{{\cal F}(\k)}
\def\b{{\cal B}}
\def\uu{{\cal U}}
\def\uuc{\uu^{c,J}_{L_0}}
\def\uuk{\uu^{k,J}_{L_0}}
\def\oo{{\cal O}}
\def\k{{\cal K}}

\def\u2{U_2({\cal H})}

\def\o2{Gr_{res}^0(p)}
\def\ii{{\cal I}}
\def\e{\epsilon}

\begin{document}

\title{\vspace*{0cm}Lagrangian Grassmannian in Infinite Dimension}


\date{}
\author{Esteban Andruchow and Gabriel Larotonda}

\maketitle

\abstract{\footnotesize{\noindent Given a complex structure $J$ on a real (finite or infinite dimensional) Hilbert space $\h$, we study the geometry of the Lagrangian Grassmannian $\Lambda(\h)$ of $\h$, i.e. the set of closed linear subspaces $L\subset \h$ such that 
$$
J(L)=L^\perp.
$$
The complex unitary group $U(\h_J)$, consisting of the elements of the orthogonal group of $\h$ which are complex linear for the given complex structure, acts transitively on $\Lambda(\h)$ and induces a natural linear connection in $\Lambda(\h)$. It is shown that any pair of Lagrangian subspaces can be joined by a geodesic of this connection. A Finsler metric can also be introduced, if one regards subspaces $L$ as projections $p_L$ (=the orthogonal projection onto $L$) or symmetries $\e_L=2p_L-I$, namely measuring tangent vectors with the operator norm. We show that for this metric the Hopf-Rinow theorem is valid in $\Lambda(\h)$: a geodesic joining a pair of Lagrangian subspaces can be chosen to be of minimal length. We extend these results to the classical Banach-Lie groups of Schatten.
}}

\setlength{\parindent}{0cm} 

\section{Introduction}
Let $\h$ be an infinite dimensional real Hilbert space with a complex structure, that is, an isometric operator $J:\h\to \h$  with $J^*=-J$ and $J^2=-I$. The (non degenerate) symplectic form is given by $w(\xi,\eta)=<J\xi,\eta>$. As usual, one defines a complex Hilbert space, denoted $\h_J$, endowing $\h$ with the complex inner product $<\xi,\eta>_J=<\xi,\eta>-i w(\xi,\eta)$. The complex structure $J$ enables one to multiply vectors in $\h$ by complex numbers in the usual way: if $z=x+iy\in{\mathbb C}$, and $\xi\in H$, then $z\xi=x\xi+yJ(\eta)$.

The purpose of this paper is the geometric study of the {\it Lagrangian Grassmannian} $\Lambda(\h)$,
the space of Lagrangian subpaces of $\h$, i.e. closed subspaces $L\subset \h$ such that
$$
J(L)=L^\perp.
$$

The Lagrangian Grassmannian $\Lambda(n)$ of $\h=\mathbb{R}^n\times \mathbb{R}^n$ ($J(x,y)=(-y,x)$) was introduced by V.I. Arnold in \cite{arnold} in 1967. He showed the connection between the topology of $\Lambda(n)\simeq O(n)/U(n)$ and the index introduced by Maslov for closed curves on a Lagrangian manifold $M\subset \mathbb{R}^{2n}$. These notions have been generalized to infinite dimensional Hilbert spaces (see \cite{furutani} and references therein), and have found several applications to Algebraic Topology, Differential Geometry and Physics. 

In this paper we consider a natural linear connection in $\Lambda(\h)$, and focus on the geodesic structure of this manifold. The Hopf-Rinow theorem states that any two points on a complete, finite dimensional Riemannian manifold can be joined by a minimal geodesic. It is well known \cite{grossman} that it is no longer true in infinite dimensions. Two points may not be even joined by a geodesic \cite{atkin}. The manifold $\Lambda(\h)$ is not even Riemannian, the natural metric available for the tangent spaces, as it will be clear, is the usual norm of operators in the Hilbert space $\h$. Our main result here shows that any two Lagrangian subspaces $L_0,L_1\subset \h$ can be joined by a minimal geodesic. Denote by $p_L$ the orthogonal projection onto $L$. In general two projections $p_{L_0}$, $p_{L_1}$ verify $\|p_{L_0}-p_{L_1}\|\le 1$. We show that if $\|p_{L_0}-p_{L_1}\|=1$ there can be two or infinite many minimal geodesics of $\Lambda(\h)$ joining $L_0$ and $L_1$. If $\|p_{L_0}-p_{L_1}\|<1$, the minimal geodesic is unique.

The {\it real Grassmannian}, the space of closed subspaces of $\h$,  will be denoted $Gr(\h)$. 
The space of closed complex subspaces, i.e. subspaces $S\subset \h$ such that $z\xi\in S$ whenever $z\in {\mathbb C}$ and $\xi\in S$, or equivalently
$$
J(L)=L,
$$
will be denoted the {\it complex Grassmannian } $Gr(\h_J)$.

Clearly $Gr(\h_J)\subset Gr(\h)$ and $\Lambda_J(\h)\subset Gr(\h)$. It is known (see \cite{cpr,furutani}) that the three sets are differentiable manifolds, and that the inclusions are submanifofds. Also it is clear that $Gr(\h_J)\cap\Lambda_J(\h)=\emptyset$. 

If $S\in Gr(\h)$,  then clearly $S\in Gr(\h_J)$ if and only if $p_SJ=Jp_S$. Also it is clear in this case that $p_S$ is the $<\ ,\ >_J$-orthogonal projection onto $S$. More generally, if $\b(\h)$ denotes the space of (real) linear operators in $\h$, the space $\b(\h_J)$ of complex linear operators consists of all elements $a\in\b(\h)$ such that $aJ=Ja$. Moreover, $L\in Gr(\h)$ is a Lagrangian subspace if and only if $p_LJ+Jp_L=J$.

It is customary to parametrize closed subspaces via orthogonal projections, $S\leftrightarrow p_S$, in order to carry on geometric or analytic computations. We shall also consider here, as in \cite{pr}, an alternative description using symmetries. Denote by $\epsilon_S=2p_S-I$, i.e. the symmetric orthogonal transformation which acts as the identity in $S$ and minus the identity in $S^\perp$. Therefore we identify
$$
S\leftrightarrow p_S \leftrightarrow \epsilon_S,
$$
and shall favor which suits best in each situation.
With these identifications, one has that
$$
Gr(\h_J)=\{\epsilon=\epsilon_S: \epsilon^*=\epsilon=\epsilon^{-1} \hbox{ and } \epsilon J=J\epsilon\},
$$
and 
$$
\Lambda_J(\h)=\{\epsilon=\epsilon_L: \epsilon^*=\epsilon=\epsilon^{-1} \hbox{ and } \epsilon J=-J\epsilon\}.
$$
That is complex subspaces commute with $J$ whereas Lagrangian subspaces anti-commute with $J$.

Denote by $O(\h)$ the orthogonal group of $\h$, i.e. $$
O(\h)=\{g\in\b(\h): g^*=g^{-1}\}.
$$
Here $g^*$ denotes the transpose of $g$, we shall use the same notation for the adjoint for complex linear operators, and no confusion should arise. Clearly $O(\h)$ acts on $Gr(\h)$ by means of $g\cdot S=g(S)$, or equivalently, $g\cdot \epsilon_S=g\epsilon_S g^*$. Denote by $U(\h_J)\subset O(\h)$ the subgroup 
$$
U(\h_J)=\{u\in O(\h): uJ=Ju\},
$$
i.e. the unitary group of the complex Hilbert space $\h_J$. This group $U(\h_J)$ acts both on $Gr(\h_J)$ and $\Lambda_J(\h)$ in the same fashion. 
The actions of $O(\h)$ on $Gr(\h)$ and of $U(\h_J)$ on $Gr(\h_J)$ are locally transitive: if $\epsilon_1$ and $\epsilon_2$ are symmetries in the same Grassmannian, such that $\|\epsilon_1-\epsilon_2\|<2$, then they are conjugate by an element of the group. Therefore the orbits of the corresponding actions coincide with the connected components. For the case of the real and complex Grassmannians, the components are parametrized by the dimensions of the subspaces and their complements. For the Lagrangian Grassmannian, the action of $U(\h_J)$ is transitive, and in particular $\Lambda_J(\h)$ is connected. 

In this paper we introduce a linear connection (Section 2) and a Finsler metric (Section 3) in $\Lambda(\h)$. The linear connection is the one defined in \cite{pr} and \cite{cpr} for the whole Grassmannian $Gr(\h)$, which restricts to $\Lambda(\h)$: if $Y$ is a  field  and $X$ is a vector both tangent to $\Lambda(\h)$, then the derivative $\nabla_XY$ performed in $Gr(\h)$ remains tangent to $\Lambda(\h)$. The geodesics of $\Lambda(\h)$ can therefore be computed. We show that the exponential map of $\Lambda(\h)$ is onto (a property which in general $Gr(\h)$ does no have), and thus any pair of Lagrangian subspaces can be joined by a geodesic curve. Moreover, we show that the geodesic can be chosen of minimal length for the Finsler metric given by the usual operator norm at each tangent space of $\Lambda(\h)$. In other words, the Hopf-Rinow theorem is valid in $\Lambda(\h)$ for this metric. We also consider the geometry of certain subsets of  $\Lambda(\h)$. In Section 4 we consider the graphs of unbounded self-adjoint Fredholm operators, which form an open subset of $\Lambda(\h)$. In Section 5 we study the submanifolds obtained as orbits of the Fredholm unitary group $U_c(\h_J)$,
$$
U_c(\h_J)=\{u\in U(\h_J): u-I \hbox{ is compact}\}.
$$
These are shown to be geodesically convex: if two Lagrangian subspaces lie in the same orbit, then the minimal geodesics which join them in $\Lambda(\h)$ remain inside the orbit. With the same technique, we treat orbits of a subspace under the action of the $k$-Schatten unitary groups $U_k(\h_J)$,
$$
U_k(\h_J)=\{u\in U(\h_J): u-I\in B_k(\h_J)\},
$$
where $B_k(\h_J)$ denotes the $k$-Schatten ideal of $\h_J$, for $k\ge 2$. Results analogous to the compact case are obtained, the Finsler metric considered here is the one induced by the $k$-norm at every tangent space. 

\section{Linear connection}

In \cite{pr} Porta and Recht introduced a linear connection in the Grassmannian  of a (real or complex) C$^*$-algebra. This connection was obtained from a natural reductive structure on this homogeneous space. Given the action of the orthogonal (or unitary group), the orbit of a given element under the action can be regarded as a quotient (or homogeneous space) of this group, by the isotropy group, or subgroup of transformations which leave the element fixed. Therefore the tangent space of the Grassmannian at this element is isomorphic to the quotient of the corresponding Banach-Lie algebras. A reductive structure is a smooth choice of invariant suplements of the Lie algebra of the isotropy groups inside the Lie algebra of the orthogonal (or unitary) group. The tangent spaces of the Grassmannian are naturally isomorphic to these supplements, and the main data of the reductive structure are precisely these isomorphisms, usually called the $1$-form of the reductive structure. The supplement must fulfill two requirements: first it has to be invariant under the adjoint representation of the isotropy subgroup, second, supplements have to vary smoothly as the element varies. In the  case discused here, the real Grassmannian $Gr(\h)$, given an element $\epsilon_0\in Gr(\h)$, the isotropy subgroup is
$$
G_{\epsilon_0}=\{g\in O(\h): g\epsilon_0=\epsilon_0 g\}.
$$
The Lie algebra of $O(\h)$ is $\b_{as}(\h)=\{a\in\b(\h): a^*=-a\}$, the space of anti-symmetric operators in $\h$. Therefore, the Lie algebra of the isotropy group is 
$$
\ii_{\epsilon_0}=\{b\in\b_{as}(\h): b\epsilon_0=\epsilon_0 b\}.
$$
If $S_0$ is the subspace corresponding to the symmetry $\epsilon_0$, then elements in $\ii_{\epsilon_0}$ can be writen as block matrices of the form
$$
\left(
\begin{array}{cc}
b_{11} & 0 \\
0 & b_{22} 
\end{array}
\right) \begin{array}{ll} S_0 \\ S_0^{\perp} \end{array},
$$
where $b_{11}$ and $b_{22}$ are anti-symmetric operators in $S_0$ and $S_0^\perp$, respectively.
The natural choice for a supplement for $\ii_{\epsilon_0}$ done in \cite{pr}, was to take all anti-symmetric  operators of the form
$$
\left(
\begin{array}{cc}
0 & c_{12} \\
-c_{12}^* & 0 
\end{array}
\right) \begin{array}{ll} S_0 \\ S_0^{\perp} \end{array}.
$$
If one denotes by $H_{\epsilon_0}$ this space, then
$$
H_{\epsilon_0}=\{c\in\b_{as}(\h): c\epsilon_0=-\epsilon_0 c\}.
$$
A reductive structure induces a linear connection in $Gr(\h)$ in a standard fashion. The main invariants of this connection can be described: geodesics, exponential map, curvature and torsion.
First let us characterize the tangent spaces of $Gr(\h)$. Since $Gr(\h)$ is a submanifold of both $\b_s(\h)$, the space of symmetric operators of $\h$ and of $O(\h)$ (symmetries are orthogonal transformations), then $(T Gr(\h))_{\epsilon_0}\subset \b_s(\h) \cap [\epsilon_0 \b_{as}(\h)]$. In fact, as it is shown in \cite{pr}, equality holds. Thus an element  $x\in(T Gr(\h))_{\epsilon_0}$ verifies $x^*=x$ and $(\epsilon_0 x)^*=-\epsilon_0 x$. That is
$$
(T Gr(\h))_{\epsilon_0}=\{x\in\b_s(\h): x\epsilon_0=-\epsilon_0 x\}.
$$
In other words, in terms of the decomposition $\h=S_0\oplus S_0^\perp$, tangent vectors correspond to symmetric co-diagonal matrices. 

Symmetries in $Gr(\h_J)$ commute with $J$. Therefore if $\epsilon_0\in Gr(\h_J)$, then
\begin{equation}\label{tangentecompleja}
(T Gr(\h_J))_{\epsilon_0}=\{x\in \b_s(\h): x\epsilon_0=-\epsilon_0 x \hbox{ and } xJ=Jx\}.
\end{equation}
Analogously, if $\epsilon_0\in\Lambda_J(\h)$,
\begin{equation}\label{tangentelagrangiana}
(T \Lambda_J(\h))_{\epsilon_0}=\{x\in \b_s(\h): x\epsilon_0=-\epsilon_0 x \hbox{ and } xJ=-Jx\}.
\end{equation}

If $\epsilon=2p-I\in Gr(\h)$, let us denote by $\Pi_\epsilon:\b_s(\h)\to \b_s(\h)$ the natural projection onto co-diagonal matrices with respect to $S=R(p)$:
\begin{equation}\label{proyeccion1}
\Pi_\epsilon(a)=(I-p)a p+p a (I-p).
\end{equation}
Using that $p=\frac12(\epsilon+I)$, after elementary computations, one obtains the alternative formula
\begin{equation}\label{proyeccion2}
\Pi_\epsilon(a)=\frac12(a-\epsilon a \epsilon).
\end{equation}

Let us write the formula for the covariant derivative. Suppose that $\gamma(t)$ is a smooth curve in $Gr(\h)$, and $X(t)$ is a tangent field along $\gamma(t)$, then
\begin{equation}\label{derivada covariante}
\frac{D}{d t} X=\Pi_{\epsilon(t)}(\dot{X}(t)).
\end{equation}
The $1$-form of the reductive structure has also a simple formula. For instance the isomorphism between $(T Gr(\h))_{\epsilon_0}$ and $H_{\epsilon_0}$ is given by
$$
x=\left( \begin{array}{cc} 0 & x_{12} \\ x_{12}^* & 0 \end{array} \right)  \mapsto \tilde{x}=\left( \begin{array}{cc} 0 & x_{12} \\ -x_{12}^* & 0 \end{array} \right).
$$
If $x_0\in (T Gr(\h))_{\epsilon_0}$, the unique geodesic $\delta$ of the connection with $\delta(0)=\epsilon_0$ and $\dot{\delta}(0)=x_0$ is given by
\begin{equation}\label{geodesica1}
\delta(t)= e^{t \tilde{x}_0} \epsilon_0 e^{-t\tilde{x}_0}.
\end{equation}
Note the remarkable fact that since $\tilde{x}_0$ anti-commutes with $\epsilon_0$ one has that $e^{t\tilde{x}_0}\epsilon_0=\epsilon_0 e^{-t\tilde{x}_0}$, and therefore
\begin{equation}\label{geodesica2}
\delta(t)=e^{2t\tilde{x}_0}\epsilon_0=\epsilon_0 e^{-2t\tilde{x}_0}.
\end{equation}
The main result concerning the existence of geodesics joining given points in $Gr(\h)$ is the following (see \cite{brown,cpr,pr}):
\begin{thm}
If $\epsilon_0, \epsilon_1\in Gr(\h)$ verify $\|\epsilon_0-\epsilon_1\|<2$, then there exists a unique geodesic $\delta(t)=e^{2tz}\epsilon_0$ of this connection with $\delta(0)=\epsilon_0$ and $\delta(1)=\epsilon_1$. Moreover, $\|z\|<\pi/2$.
\end{thm}
Note that in general two elements in $Gr(\h)$ lie at most at norm distance $2$. 

We shall use this connection in $Gr(\h)$ to induce a connection in the submanifold $\Lambda_J(\h)$. The  connection projects well in these submanifolds (the submanifolds $Gr(\h_J)$ and $\Lambda_J(\h)$ are flat). We only prove this fact for the Lagrangian Grassmannian. 
\begin{lem}
If $\epsilon$ lies in $\Lambda_J(\h)$ and $x$ anti-commutes with $J$, then $\Pi_\epsilon(x)$ anti-commutes with $J$
\end{lem}
\begin{proof}
Recall expression (\ref{proyeccion2}) for $\Pi_\epsilon$, $\Pi_\epsilon(x)=\frac12(x-\epsilon x\epsilon)$. Since $\epsilon$ and $x$ anti-commute with $J$, then 
$$
J\Pi_\epsilon(x)=\frac12(Jx-J\epsilon x\epsilon)=\frac12(-xJ+\epsilon x\epsilon J)=-\Pi_\epsilon(x)J.
$$
\end{proof}

\begin{prop}
Suppose that $\gamma(t)$ is a curve in  $\Lambda_J(\h)$, and $X(t)$ is a field tangent to  $\Lambda_J(\h)$ along $\gamma(t)$. Then $\frac{D}{dt}X\in (T\Lambda_J(\h))_{\gamma(t)}$.

Moreover, the submanifold  $\Lambda_J(\h)$ is geodesically convex in the following sense: if $\epsilon_0,\epsilon_1\in  \Lambda_J(\h)$ with $\|\epsilon_0-\epsilon_1\|<2$, then the unique geodesic $\delta$ joining them in $Gr(\h)$ lies in fact in  $\Lambda_J(\h)$.
\end{prop}
\begin{proof}
If $\gamma(t)\in \Lambda_J(\h)$ and $X(t)$ is tangent to $\Lambda_J(\h)$, then $\gamma(t)$ and $X(t)$ anti-commute with $J$ for all $t$. In particular, $\dot{X}(t)$ anti-commutes with $J$ for all $t$. Using the lemma above, it follows that $\Pi_{\gamma(t)}(\dot{X}(t))$ anti-commutes with $J$ for all $t$.

Suppose that $\epsilon_0, \epsilon_1\in \Lambda_J(\h)$ verify $\|\epsilon_0-\epsilon_1\|<2$. Let $z\in\b_{as}(\h)$ with $z\epsilon_0=-\epsilon_0 z$ and $\|z\|<\pi/2$, such that $\delta(t)=e^{2tz}\epsilon_0$ is the unique geodesic in $Gr(\h)$ joining $\epsilon_0$ and $\epsilon_1$. Then, since $\e_1$ and $\e_2$ anti-commute with $J$, $\epsilon_1\epsilon_0=e^{2z}$ commutes with $J$. The fact that $\|2z\|<\pi$, implies that $2z$ is a series in powers of $u^{2z}$ and therefore also commutes with $J$. This implies that the geodesic $\delta$ lies in $\Lambda_J(\h)$. 
\end{proof}
That symmetries lying at distance less than $2$ are joined by a geodesic  in all three manifolds considered here is more or less known or natural. The interesting question is wether symmetries at distance $2$ can be connected. Let $\epsilon_0$ and $\epsilon_1$ be two symmetries with $\|\epsilon_0-\epsilon_1\|=2$, and let $S_0$ and $S_1$ be the subspaces that they represent.
When considering the unitary equivalence of subspaces, it is natural to consider the following reducing subspaces. Denote
$$
\h_{11}=S_0\cap S_1 \ , \ \h_{00}=S_0^\perp\cap S_1^\perp \ , \ \h_{01}=S_0^\perp\cap S_1 \ , \ \h_{10}=S_0\cap S_1^\perp 
$$
and
$$
\h_0=(\h_{00}\oplus\h_{01}\oplus\h_{10}\oplus\h_{11})^\perp.
$$
The subspaces which simultaneously reduce $\epsilon_0$ and $\epsilon_1$ are $\h_{00}$, $\h_{11}$, $\h_{10}\oplus\h_{01}$ and $\h_0$. On $\h_{ii}$ both symmetries  coincide. On $\h_{10}\oplus\h_{01}$ they act as
$$
\epsilon_0|_{\h_{10}\oplus\h_{01}}=\left( \begin{array}{cc} 1 & \; 0 \\ 0 &  -1 \end{array} \right) \ \ \hbox{ and } \ \ \epsilon_1|_{\h_{10}\oplus\h_{01}}=\left( \begin{array}{cc} -1&\; 0 \\ 0 & 1 \end{array} \right) .
$$
On $\h_0$, $\epsilon_0$ and $\epsilon_1$ are said to be in generic position (for the reduced symmetries, the former intersections are trivial).

Clearly, in order to see if the symmetries are joined by a geodesic, it suffices to examine if they can be joined  in each of the reducing subspaces above. It is well known \cite{dixmier,halmos} (see also \cite{brown}), for the real and complex case, that the parts in generic position are conjugate; we transcribe the argument below. Also it is   known (see for instance \cite{brown}), that the remaining parts to study, i.e. the parts acting in $\h_{10}\oplus\h_{01}$ are conjugate if $\dim \h_{10}=\dim \h_{01}$. That this condition is sufficient (again, in the real and complex cases) is also well known. 

In \cite{halmos} Halmos showed that two projections in generic position are unitarily equivalent, more specifically, he showed that there exists a Hilbert space ${\cal K}$ and a unitary operator $w:\h_0\to {\cal K}\times {\cal K}$ such that 
$$
wp_0w^*=p_0'=\left( \begin{array}{cc} \;1\; & \; 0\; \\ 0 & 0 \end{array} \right) \ \ \hbox{ and } \ \ 
wp_1w^*=p'_1=\left( \begin{array}{cc} c^2 & \;cs \\ cs & s^2 \end{array} \right),
$$
where $c,s$ are positive commuting contractions acting in ${\cal K}$ and satisfying $c^2+s^2=1$, and $\ker c=\ker s=\{0\}$. Then there exists an anti-symmetric operator $y$ acting on ${\cal K}\times {\cal K}$, which is a co-diagonal matrix, and  such that $e^yp'_0e^{-y}=p'_1$. In that case, the element $z_0=w^*yw$ is an anti-symmetric operator in $\h_0$, which verifies $e^{z_0}p_0e^{-z_0}=p_1$, and is co-diagonal with respect to $p_0$ (or equivalently, anti-commutes with $\epsilon_0$). Moreover, note that  $\|y\|\le \pi/2$, so that  $\|z_0\|\le\pi/2$. Let us prove these facts. By a functional calculus argument, there exists a self-adjoint operator $x$ in the  with $\|x\|\le \pi/2$,  such that $c=\cos(x)$ and $s=\sin(x)$. Consider
$$
y=\left( \begin{array}{cc} 0 & -x \\ x & \; 0 \end{array} \right)
$$
Clearly $y^*=-y$, $\|y\|\le \pi/2$. A straightforward computation shows that 
$$
e^yp'_0e^{-y}=p'_1.
$$
We shall call this construction Halmos' trick.

Let us consider now the Lagrangian Grassmannian.
\begin{thm}\label{hopfrinowlagrange}
If $\epsilon_0,\epsilon_1\in\Lambda_J(\h)$, then there exists a geodesic joining them in $\Lambda_J(\h)$.
The geodesic is of the form $\delta(t)=e^{2tz}\epsilon_0$, with $zJ=Jz$  and $\|z\|\le\pi/2$.
If $\|\e_0-\e_1\|<2$ the geodesic is unique.

\end{thm}
\begin{proof}
Let $\epsilon_0, \epsilon_1\in \Lambda_J (\h)$ correspond to Lagrangian subspaces $L_0$ and $L_1$. Note that $J(L_0)=L_0^\perp$,  $J(L_1)=L_1^\perp$, and that $J$ is an orthogonal transformation, imply that
\begin{equation}\label{Jsereduce}
J(\h_{00})=\h_{11} \ , \ J(\h_{11})=\h_{00} \ , \ J(\h_{01})=\h_{10}  \ \hbox{ and } \ J(\h_{10})=\h_{01}.
\end{equation}
Therefore $J(\h_0)=\h_0$. Then it is also clear that $L_0\cap \h_0$ and $L_1\cap \h_0$ are Lagrangian subspaces of $\h_0$ (corresponding to the complex structure $J|_{\h_0}$ in $\h_0$), which are represented by the symmetries $\e'_0$ and $\e'_1$, the reductions of $\e_0$ and $\e_1$ to $\h_0$. Another consequence of (\ref{Jsereduce}), is that the operator 
$$
z_2=\frac{\pi}{2} \ J|_{\h_{01}\oplus\h_{10}}:\h_{01}\oplus\h_{10}\to \h_{01}\oplus\h_{10}
$$
implements the unitary equivalence between the parts of $\e_0$ and $\e_1$ to $\h_{01}\oplus\h_{10}$, and clearly $\|z_2\|=\pi/2$. Indeed, first note that $J|_{\h_{01}\oplus\h_{10}}$ is a complex operator because it commutes with $J$. Next, since
 $v=J|_{\h_{10}}:\h_{10}\to \h_{01}$ is a surjective isometry, then 
$$
u: \h_{01}\oplus \h_{10}\to \h_{01}\oplus \h_{10} \ , u(\xi'+\xi'')=v^*\xi'+v\xi'',
$$
is a unitary operator. In matrix form (in terms of the decomposition $\h_{01}\oplus \h_{10}$), 
$$
u=
\left(
\begin{array}{ll}
0 & v \\ v^* & 0
\end{array}
\right).
$$
Apparently, $u p|_{\h_{01}\oplus \h_{10}} u^*=q_1|_{\h_{01}\oplus \h_{10}}$. Finally, a straightforward matrix computation shows that $e^{z_2}=u$.

Thus we have to prove that the reductions $\e'_0$ and $\e'_1$ to the generic part are conjugate. Using Halmos' trick, we may consider the projections 
$$
p''_0= \left( \begin{array}{rc} 1 & \;0 \\ 0 & 0 \end{array} \right) \ \ \hbox{ and } \ \ 
p''_1=\left( \begin{array}{cc} c^2 & cs \\ cs & s^2 \end{array} \right) .
$$
These projections are obtained from the original projections given by $\e_0,\e_1$ and unitary equivalence with a unitary operator $w:\h_0\to {\cal K}\times {\cal K}$. Denote by $J'$ the corresponding complex structure in ${\cal K}\times {\cal K}$: $J'=w J|_{\h_{01}\oplus\h_{10}} w^*$. Let $\e''_0$ and $\e''_1$ the symmetries (in $\Lambda_J({\cal K}\times {\cal K})$ of the complex structure $J'$) given by $p''_0$ and $p''_1$. 
First note that $\e''_0$ is a $J'$ Lagrangian means that 
$$
\left( \begin{array}{cc} 1 & \; 0 \\ 0 & -1 \end{array} \right)  \-\hbox{ and } \ J'
$$
commute, which implies that
$$
J'=\left( \begin{array}{cc} 0 & j \\ -j^* & 0 \end{array} \right)
$$
with $j$ a unitary operator in ${\cal K}$. Also $J'$ anti-commutes with 
$$
\e''_1=\left( \begin{array}{cc} 2c^2-1 & 2cs \\ 2cs & 2s^2-1 \end{array} \right).
$$ 
After straightforward matrix computations one obtains that this implies
$$
\left\{ \begin{array}
csj^*= jcs \\
(2c^2-1)j=-j(2s^2-1) \end{array} \right.
$$
The second equation gives $j^*(2c^2-1)j=-2s^2+1$, or equivalently, $(j^*cj)^2+s^2=1$. This relation together with $c^2+s^2=1$ imply that $(j^*cj)^2=c^2$, and therefore $c$ and $j$ commute, because $c\ge 0$. Then $s$ also commutes with $j$. The first equation  can thus be rewritten $csj^*=cjs$, and since $\ker c=\{0\}$ one has $sj^*=js$.
Using these  relations, a straightforward computation shows that 
$$
\left( \begin{array}{cc} c & -s \\ s & \; c \end{array} \right)\  \hbox{ commutes with } \ J'.
$$
Note that since $0\le c,s\le 1$, $cs=sc$ and $c^2+s^2=1$, there exists $0\le x\in \b({\cal K})$ with $\|x\|\le \pi/2$ such that $c=\cos(x)$ and $s=\sin(x)$. 
The right hand matrix above is the exponential of 
$$
z'_0=\left( \begin{array}{cc} 0 & -x \\ x & \;0 \end{array} \right),
$$ 
Moreover, since $c,s$ commute with $j$, $x$ commutes with $J'$ and therefore $z'_0$ is a complex operator.
\end{proof}

\section{Finsler metric}

In \cite{pr} Porta and Recht showed that the unique minimal geodesic joining two elements $\e_0 , \e_1$ in $Gr(\h)$ or $Gr(\h_J)$ with $\|\e_0-\e_1\|<2$, has minimal length along its path if one considers these manifolds with the Finsler metric which consists of endowing each tangent space with the usual operator norm. That is, the length $L(\gamma)$ of a curve $\gamma$ parametrized in the interval $[t_0,t_1]$ is measured by
$$
L(\gamma)=\int_{t_0}^{t_1} \|\dot{\gamma}(t)\| d t .
$$
For $M=Gr(\h)$, $Gr(\h_J)$ or $\Lambda_J(\h)$, denote by $d$ the rectifiable metric
$$
d(\e_0,\e_1)=\inf\{L(\gamma): \gamma \hbox{ joins } \e_0 \hbox{ and } \e_1 \hbox{ in } M\}.
$$
If $S_0$ (resp. $S_1$) is the subspace given by $\e_0$ (resp. $\e_1$), in the previous section it was shown that if $\|e_0-\e_1\|=2$ and $\dim(S_0\cap S_1^\perp)=\dim(s_0^\perp\cap S_1)$, then there exists a geodesic $\delta(t)=e^{2tz}\e_0$ which joins them, with $\|z\|=\pi/2$. Let us show that this geodesic is also minimal along its path. 

\begin{prop}
Let $\e_0,\e_1\in Gr(\h)$ (resp. $Gr(\h_J)$) with $\|\e_0-\e_1\|=2$ and $\dim (S_0\cap S_1^\perp)=\dim (S_0^\perp\cap S_1)$. Then the geodesic $\delta(t)=e^{2tz}\e_0$ which joins $\e_0$ and $\e_1$  has minimal length along its path in the interval $[0,1]$.
 \end{prop}
 \begin{proof}
The geodesic that joins $\e_0$ and $\e_1$ verifies $\|z\|=\pi/2$. Then, if $[t_1,t_2]$ is a proper sub-interval of $[0,1]$, then $\delta|_{[t_0,t_1]}$ is the unique geodesic joining them because $\|\delta(t_0)-\delta(t_1)\|<2$, and it has therefore minimal length. Therefore it remains to examine the case $t_0=0$ and $t_1=1$. Suppose that there exists a curve $\gamma$ with $L(\gamma)<L(\delta)+r$, ($r>0$) with $\gamma(0)=\e_0$ and $\gamma(1)=\e_1$. There exists $t_r\in [0,1)$ such that if $\delta_r$ denotes the curve $\delta|_{[0,t_r]}$ and $\delta^r$  denotes $\delta|_{[t_r,1]}$, then $L(\delta^r)<r/2$. Note that $\delta_r$ has minimal length joining $\epsilon_0$ and $\delta(t_r)$.  Then if $\nu$ denotes the curve $\gamma$ followed by $\delta^r$, one has that $\nu$ joins $\epsilon_0$ and $\delta(t_r)$, with
 $$
 L(\nu)=L(\gamma)+L(\delta^r)\le L(\delta)-r+L(\delta^r)=L(\delta_r)-r+2 L(\delta_r)<L((\delta_r),
 $$
 a contradiction. 
 \end{proof}
There is a partial converse to this fact, which is a direct consequence of Theorem 9 in Brown's paper \cite{brown}. Brown proves that if, in the notation of the previous section, $\dim(\h_{01})\ne \dim(\h_{10})$ and $d(\e_0,\e_1)=\pi$, then $\e_0$ and $\e_1$ are not joined by a minimizing path.
 \begin{prop}
Suppose that there exists a geodesic $\delta(t)=e^{2tz}\e_0$ joining $\e_0$ and $\e_1$ with $\|z\|=\pi/2$. Then $\dim(\h_{01})= \dim(\h_{10})$.
 \end{prop}
 \begin{proof}
By the previous proposition, $\delta$ is minimizing, with $L(\delta)=\pi$. Thus $d(\e_0,\e_1)=\pi$, and therefore by Browns's result, $\dim(\h_{01})= \dim(\h_{10})$.
 \end{proof}
 
For the Lagrangian Grassmannian one has the following stronger result:
\begin{thm}
Let $\e_0,\e_1\in \Lambda_J(\h)$. Then there exists a minimal geodesic joining them. The geodesic is unique if $\|\e_0-\e_1\|<2$.
\end{thm}
\begin{proof}
From Theorem \ref{hopfrinowlagrange} we know that there exists a geodesic $\delta$ with $\|z\|\le \pi/2$. By the proposition above, it is minimal in $Gr(\h)$, a fortiori it is minimal in $\Lambda_J(\h)$.
\end{proof}

\begin{exmp}
Consider $\k$ a Hilbert space and let $\h=\k\times \k$ with the usual Hilbert space inner product. Let $J:\h\to\h$ be given by $J(\xi,\eta)=(-\eta,\xi)$. Then it is well known that graphs of symmetric operators $a:\k\to \k$ are Lagrangian subspaces of $\h$. One may even consider unbounded symmetric operators, and a straightforward argument shows that 
$$
a:D(a)\subset\k\to \k
$$
is closed and symmetric and $G_a$ denotes its graph, then $G_a\in\Lambda_J(\h)$ if and only if $D(a)=D(a^*)$, i.e. $a$ is self-adjoint. The above results show that for any pair of self-adjoint operators $a,b$ there exists an anti-hermitic complex operator $z$ in $\h_J$ with $\|z\|\le \pi/2$ such that $e^{z}(G_a)=G_b$, and the curve $e^{tz}(G_a)$ is a minimal geodesic in sense mentioned before. Note that the fact that this curve is a geodesic implies that $z$ anti-commutes with the symmetry $\e_{G_a}=2p_{G_a}-I$ corresponding to the subspace $G_a$ ($p_{G_a}$ denotes the orthogonal projection onto $G_a$).   
Another example of a Lagrangian subspace, which is not of this type above, is $L_0=\{0\}\times \k$.
In some of these examples it is easy  to characterize the operator $z$. 
\begin{enumerate}
\item
Consider $L_0$ as above and $b:D(b)\subset\k\to \k$ an arbitrary self-adjoint operator. Then $\e_{L_0}$ is given by the matrix (in terms of the decomposition $\h=\k\times\k$)
$$
\e_{L_0}=\left( \begin{array}{cc} -I & \;0 \\ 0  & I  \end{array} \right).
$$
Note also that $J$ is, in matrix form
$$
J=\left( \begin{array}{cc} 0 & -I \\ I  & \; 0  \end{array} \right).
$$
The operator $z$ which is the velocity of a minimal geodesic joining $L_0$ and $G_b$ must therefore verify
$$
zJ=Jz \ \ \hbox{ and } \ \ z\e_{L_0}=-e_{L_0} z.
$$
These imply that $z$ has matrix form
$$
z=\left( \begin{array}{cc} 0 & x \\ -x  & \;0  \end{array} \right),
$$
with $x$ symmetric in $\k$, and $\|x\|\le \pi/2$. Then $e^z$ can be explicitely computed, namely
$$
e^z=\left( \begin{array}{cc} \cos(x) & \sin(x) \\ -\sin(x)  & \cos(x)  \end{array} \right).
$$
Then the fact that $e^z(L_0)=G_b$ means the following: for each $\eta\in D(b)$ there exists $\xi\in \k$ such that $\sin(x)\xi=\eta$ and $\cos(x)\xi=b\eta$. In particular, $\sin(x)$ is injective with dense range equal to $D(b)$ and
$$
b \sin(x)=\cos(x) \ , \ \ \hbox{ or equivalently } b= \cos(x) (\sin(x))^{-1},
$$
with $(\sin(x))^{-1}:D(b)\to \k$ the (unbounded) inverse of $\sin(x)$.
The element $x$ is unique if and only in $\|p_{L_0}-p_{G_b}\|<1$, a fact which is apparently equivalent to $b$ being invertible.
\item
Consider $a=I$ and $b$ arbitrary. Then a straightforward computation shows that
$$
p_{G_I}=\frac12 \left( \begin{array}{cc}  I &  \; I \\  I  &  I  \end{array} \right)  \ \ \hbox{ and }
\ \ \e_{G_I}=\left( \begin{array}{cc} 0 & I \\ I  & \; 0  \end{array} \right).
$$
Again in this case the operator $z$ is of the codiagonal  form described above. In this case we obtain that for any $\eta\in D(b)$ there exists $\xi\in\k$ such that $\sin(x)\xi+\cos(x)\xi=\eta$ (implying that $\sin(x)+\cos(x)$ is injective with closed range $D(b)$) and $-\sin(x)\xi=\cos(x)\xi=b\eta$, and in particular
$$
b=(-\sin(x)+\cos(x))(\sin(x)+\cos(x))^{-1}.
$$
\item
Fix a symmetry $e\in \b(\k)$ ($e^*=e$, $e^2=1$). Elementary calculations show that $p_{G_e}$ and $\e_{G_e}$ are given by
$$
p_{G_e}=\frac12\left( \begin{array}{cc}  I & \; e \\ e  &  I  \end{array} \right)  \ \ \hbox{ and }
\ \ \e_{G_e}=\left( \begin{array}{cc} 0 & \;e \\ e  & 0  \end{array} \right).
$$
There is a geodesic connecting $G_I$ and $G_e$, which by the previous item is given by an operator $z$ in $\h$ of the form
$$
z=\left( \begin{array}{cc}  0 & x \\ -x  &  0  \end{array} \right),
$$
with $x$ self-adjoint in $\k$. By comparing the final point of the geodesic $\delta(t)=e^{2tz}\e_{G_I}$ with $\e_{G_e}$, one obtains that
$$
\cos(2x)=e \ \ \hbox{ and } \sin(2x)=0,
$$
which means that $x$ is of the form $x=\frac{\pi}{2}(p_+-p_-)$ where $p_+$ and $p_-$ are mutually orthogonal projections in $\k$ which decompose $\frac12(I-e)$ (the projection onto the spectral subspace of $e$ corresponding to the eigenvalue $-1$): 
$$
p_++p_-=\frac12(I-e).
$$
Conversely, any decomposition as this one provides an element $z$ such that $e^{2z}\e_{G_I}=\e_{G_e}$. This proves that there are infinitely many minimal geodesics joining $G_I$ and $G_e$ in $\Lambda_J(\h)$. Note, accordingly, that 
$\|p_{G_I}-p_{G_e}\|=1$, because $I-e$ is non invertible in $\k$.

\end{enumerate}

\end{exmp}
\section{Unbounded self-adjoint Fredholm operators}
In this section we examine the set of subspaces in $\Lambda_J(\h)$ which are graphs of (unbounded) Fredholm operators. Recall that a self-adjoint operator $a:D(a)\subset \k \to \k$ is called Fredholm if the bounded operator $a(I+a^2)^{-1/2}$ is Fredholm. 

Following the notation in \cite{furutani}, if $L_0\in \Lambda_J(\h)$, denote by $\oo_{L_0}$ the set of Lagrangian subspaces which are transversal to $L_0$,
$$
\oo_{L_0}=\{S\in\Lambda_J(\h): S+L_0=\h\}.
$$
In \cite{furutani} it is proved that these sets are open in $\Lambda_J(\h)$.
As in the examples of the preceeding section, suppose $\h=\k\times \k$. 
Let us denote by $\Lambda_g(\h)$  the sets of Lagrangian subspaces which are graphs (of necessarily self-adjoint operators) acting in $\k$, and by $\Lambda_{gf}(\h)$ the subset of the latter consisting of graphs of (unbounded) Fredholm operators.

\begin{prop}\label{abierto}
$\Lambda_g(\h)$ and  $\Lambda_{gf}(\h)$ are open in $\Lambda_J(\h)$.
\end{prop}
\begin{proof}
The first fact is elementary. A subspace $S\in\Lambda_J(\h)$ is a graph of an operator $a:D(a)\subset\k\to \k$ if and only if $L\cap (\{0\}\times \k)=\{0\}$, or equivalently, $L^\perp +(\k\times\{0\})=\h$. Therefore 
$$
\Lambda_g(\h)=\{L\in\Lambda_J(\h): L^\perp + (\k\times\{0\})=\h\}=\oo_{(\k\times\{0\})}^\perp,
$$
which  is open because the map $\Lambda_J(\h)\to \Lambda_J(\h)$, $S\mapsto S^\perp$ is a homeomorphism. Indeed, in terms of projections, it is $p_S\mapsto I-p_S$.
Suppose that $L=G_a$ is the graph of $a$. It is known (and it is not difficult to see that) the matrix of the projection $p_{G_a}$ is given by
$$
p_{G_a}=\left( \begin{array}{cc} (I+a^2)^{-1} & a(I+a^2)^{-1} \\ a(I+a^2)^{-1} & a^2(I+a^2)^{-1} \end{array} \right).
$$
Note that $a$ is Fredholm if and only if the $2,2$-entry in this matrix is a Fredholm operator. Indeed, if $a(I+a^2)^{-1/2}$ is Fredholm, then so is 
$$
a^2(I+a^2)^{-1}=(a(I+a^2)^{-1/2})^2.
$$
On the other hand, if the $2,2$ entry $a^2(I+a^2)^{-1}$ is Fredholm, then  its square root $|a|(I+a^2)^{-1/2}$ is also Fredholm. In the polar decomposition $a=u|a|$, $u$ is can be chosen unitary, and therefore
$a(I+a^2)^{-1/2}=u^*|a|(I+a^2)^{-1}$ is Fredholm. Since the set of Fredholm operators of $\k$ is open, it follows that the set 
$$
\Lambda_{gf}(\h)=\{G_a\in\Lambda_g(\h): \hbox{ the } 2,2 \hbox{ entry of } p_{G_a} \hbox{ is Fredholm}\}
$$
is open in $\Lambda_g(\h)$.

\end{proof}

\begin{rem}\label{proyes}
In particular, $\Lambda_g(\h)$ and $\Lambda_{gf}(\h)$ are submanifolds of $\Lambda_J(\h)$. Consider the base point $G_I\in \Lambda_{gf}(\h)$, the graph of the identity operator of $\k$. Since $\Lambda_{gf}(\h)$ is open, any geodesic starting at $G_I$ will remain inside of $\Lambda_{gf}(\h)$ for a certain period of time. The next result examines how long is this period. First, recall the form of a (minimal) geodesic $\delta(t)=e^{tz}(G_I)$ starting at $G_I$.
Let us regard $\delta$ as a curve of projections. Then if 
$$
z=\left( \begin{array}{cc} 0 & y \\ -y & \; 0 \end{array} \right)\ \  \hbox{ then } \ \ 
e^{tz}= \left( \begin{array}{cc}  \cos(ty)& \sin(ty) \\ -\sin(ty) & \cos(ty) \end{array} \right),
$$
and
$$
p_{\delta(t)}=\frac12 \left( \begin{array}{cc}  \sin(2ty)+I & \cos (2ty) \\ \cos(2ty) & I-\sin(2ty) \end{array} \right)
$$
\end{rem}

\begin{rem}
Given a self-adjoint operator $a:D(a)\subset\k\to \k$, one can characterize the unitaries $u\in U(\h_J)$ such that $u(G_a)$ is a graph. First note that since $u$ is a unitary that commutes with $J$, it must be of the form
$$
u=\left(\begin{array}{cc} x & y \\ -y & x \end{array} \right) \hbox{ with } \left\{ \begin{array}{l}
x^*x +y^*y= xx^*+yy^*=I \\ x^*y=y^*x \end{array} \right. .
$$
Then $u(G_a)$ is a graph if and only if $x+ya$ is one to one. Indeed, if $x+ya$ is one to one, then
$(0,\eta)\in u(G_a)$ implies that there exists $\xi \in D(a)$ such that 
$$
(0,\eta)=u(\xi,a\xi)=((x+ya)\xi, (-y+xa)\xi).
$$
Since $x+ya$ is one to one, this implies that $\xi=0$ and thus $\eta=0$. Conversely, suppose that there exists $\xi_0\ne 0$ such that $(x+ya)\xi_0=0$, then clearly $\eta_0=(-y+xa)\xi_0$ satisfies that $(0,\eta_0)\in u(G_a)$. Note that $\eta_0\ne 0$: $(x+ya)\xi_0=0$ implies that 
$$
0=(x^*x+x^*ya)\xi_0=(I-y^*y+y^*xa)\xi_0,
$$
and thus $\xi_0=y^*(-y+xa)\xi_0=y^*\eta_0$. 

The self-adjoint operator whose graph is given by $u(G_a)$ is apparently
$$
(-y+ax)(x+ay)^{-1}.
$$

Accordingly, the unitaries $u\in U(\h_J)$ such that $u(G_a)$ is the graph of a Fredholm operator are the ones such that  the $2,2$-entry of $up_{G_a}u^*$ is a Fredholm operator. Namely, 
$$
(-y+ax)(I+a^2)^{-1}(-y^*+x^*a)
$$
is Fredholm. Note that it is bounded, therefore this is equivalent to 
$$
(-y+ax)(I+a^2)^{-1/2}
$$
being a Fredholm operator. In particular, if $a$ is bounded this is equivalent to $-y+ax$ being a Fredholm operator.
\end{rem}
Using the above remark, we may characterize the portion of a geodesic starting at $G_I$ that remains inside $\Lambda_{gf}$. If $x$ is a bounded operator, denote by $\sigma_p(x)$ the point spectrum of $x$ and by $\sigma_e(x)$ the essential spectrum of $x$.
\begin{lem}
let $y=y^*$ in $\k$ with $\|y\|\le \pi/2$. Then the geodesic  of $\Lambda_J(\h)$, $\delta(t)=e^{tz}(G_I)$,
where
$$
z=\left(\begin{array}{cc} 0 & y \\ -y & 0 \end{array} \right),
$$
remains inside $\Lambda_{gf}(\h)$ for all $t\in[0,1]$ if and only if
$$
\sigma_p(y)\subset (-\pi/4, \pi/2] \hbox{ and } \sigma_e(y)\subset [-\pi/2, \pi/4).
$$
\end{lem}
\begin{proof}
By the above remark and the computations done in  Proposition \ref{abierto}, $\delta(t)$ is the graph of a self-adjoint operator for $t\in[0,1]$ if and only if $\cos(ty)+\sin(ty)$ is injective. In other words, $-\pi/4$ does not belong to $\sigma_p(ty)$ for any such $t$. Apparently, since $\|y\|\le \pi/2$, this is equivalent to $\sigma_p(y)\subset (-\pi/4, \pi/2]$. On the other hand, in order that $\delta(t)$ be the graph of a Fredholm operator, the operator $I-\sin(2ty)$ must be Fredholm. That is, $\pi/2$ should not belong to $\sigma_e(2ty)$ for any $t\in[0,1]$. This is equivalent to $\sigma_e(y)\subset [-\pi/2, \pi/4)$.
\end{proof}

\begin{cor}
Let $\delta(t)=e^{tz} (G_I)$ be a geodesic of the connection, which is minimal, i.e. $\|z\|\le \pi/2$. Then $\delta(t)\in\Lambda_{gf}(\h)$ for $0\le t <\frac{\pi}{4\|z\|}$. 
\end{cor}

Let $\fred_s$ denote the set of (possibly unbounded) self-adjoint Fredholm operators of $\k$. It is a metric space with the so called gap metric, which is the metric given by  $d(f_1,f_2)=\|p_{G_{f_1}}-p_{G_{f_2}}\|$.
If $\alpha:[0,1]\to \fred_s$ is a continuous map in the gap topology, there is a topological invariant called the {\it spectral flow} associated to $\alpha$ \cite{phillips}. Let us compute this invariant for the geodesics in $\Lambda_J(\h)$ which remain inside $\Lambda_{gf}(\h)$.  Let $C$ denote the Cayley transform,
$$
C(f)=(f-iI)(f+iI)^{-1}.
$$
The spectral flow  of $\alpha$ is defined as the winding number of the curve $w(C(f))$. We shall see that the spectral flow of certain geodesics of $\Lambda_J(\h)$, starting at $G_I$, that remain inside $\Lambda_{gf}(\h)$ is trivial.
\begin{rem}
Let $\delta$ as above. Suppose that $\delta(t)\in \Lambda_{gf}(\h)$ for $t\in [0,1]$, and denote by $f(t)$ the Fredholm operator in $\k$ whose graph is $\delta(t)$. Then $-1\notin \sigma_e(C(f(t)))$ for $t\in [0,1]$.
Indeed, by the above remark and computations,  the operator $f(t)$ is given by
$$
f(t)=(-\sin(ty)+\cos(ty))(\sin(ty)+\cos(ty))^{-1}.
$$
Let us compute $C(f(t))$. By a functional calculus argument, $C(f(t))=g(ty)$, where
$$
g(x)=\left(\frac{-\sin(x)+\cos(x)}{\sin(x)+\cos(x)}-i\right)\left(\frac{-\sin(x)+\cos(x)}{\sin(x)+\cos(x)}+i\right)^{-1}.
$$
After elementary computations, one sees that $g(x)=-\sin(2x)-i\cos(2x)$. Thus
$$
C(f(t))=-i(\sin(-2ty)+\cos(-2ty))=e^{-i(\pi/2+2ty)}.
$$
By the above lemma, since $\delta(t)(G_I)$ is the graph of a Fredholm operator, $\pi/4\notin \sigma_e(ty)$ for all $t\in[0,1]$. 
Then  
$$
-\pi\notin \sigma_e(-\pi/2-2ty), 
$$
and the claim follows.
\end{rem}
We shall consider now geodesics $\delta$ with $\delta(0)=I$ which remain symmetrically inside $\Lambda_{gf}(\h)$, i.e. $\delta(t)\in \Lambda_{gf}(\h)$ for all $t$ in a symmetric interval centered at $t=0$.
\begin{prop}
With the above notations, suppose that $\delta(t)$ is a minimal geodesic which remains inside $\Lambda_{gf}(\h)$ for all $t\in[-1,1]$.
Then $-1\notin \sigma(C(f(t))$ for all $t\in[-1,1]$. In particular the winding number of $\delta(t)$, $t\in[-1,1]$, is trivial.
\end{prop}
\begin{proof}
As in the remark above $\pi/4\notin \sigma_e(ty)$. By a similar argument, reasoning with $-ty$, it follows that $-\pi/4\notin \sigma_p(-ty)$. Using the fact that for a bounded self-adjoint operator $a$, one has $\sigma(a)=\sigma_e(a)\cup \sigma_p(a)$, it follows that $\pi/4 \notin \sigma(ty)$ for all $t\in[-1,1]$. 
Therefore, $-1\notin \sigma(C(f(t))$ for all $t\in[-1,1]$. The set $\{u\in U(\k): -1\notin \sigma(u)\}$ is contractible (it is homeomorphic to the ball $\{x\in \b(\k): x^*=x \hbox{ and } \|x\|<\pi\}$ via the exponential map $x\mapsto e^{ix}$). Therefore the assertion on the triviality of the winding number follows.
\end{proof}

\section{Orbits of the Fredholm group}

In this section we shall consider a fixed Lagrangian subspace $L_0\subset \h$ and consider its orbit under the action of the complex Fredholm group
$$
U_c(\h_J)=\{u \in U(\h_J): u-I \hbox{ is compact}\}.
$$
Denote by ${\cal C}(\h)$ the ideal of compact operators in $\h$ (and accordingly, ${\cal C}(\h_J)$ the closed and complemented subspace of complex compact operators).
In \cite{furutani} K. Furutani introduced the notion of Fredholm pairs of Lagrangian subspaces. A pair  $(S_1,S_2)$ of elements in $\Lambda_J(\h)$ is a Fredholm pair if $S_1\cap S_2$ and $S_1^\perp \cap S_2^\perp$ are finite dimensional. Equivalently, $p_{S_1^\perp}|_{S_2}:S_2\to S_1^\perp$ is a Fredholm operator. He showed that if there exists $u\in U_c(\h_J)$ such that $u(L_0^\perp)=S$, then $(L_0^\perp,S)$ is a Fredholm pair. He also showed that the set of all $S$ such that $(L_0^\perp,S)$ is a Fredholm pair is an open subset of $\Lambda_J(\h)$ (containing $L_0$). In particular this implies, as we shall see below, that not every such $S$ is of the form $u(L_0^\perp)$ for some $u\in U_c(\h_J)$.

On the other hand, if one considers the {\it real } Fredholm group (i.e. operators $w$ in $O(\h)$ such that $w-I$ is compact), it is known \cite{pressleysegal,segalwilson,voiculescu} that the set of all such $w(L_0)$ form a differentiable manifold, known as the restricted or Sato Grassmannian $G_{res}^0(L_0)$ \cite{sato} (or more precisely, the connected component of $L_0$ in the restricted Grassmannian). It consists of all real subspaces $T\subset \h$ (Lagrangian or not) such that
\begin{enumerate}
\item
$p_{L_0}|_{T}:T\to L_0 \hbox{ is Fredholm (of index) } 0,$
and 
\item
$p_{L_0}|_{T}:T\to L_0 \hbox{ is compact}.$
\end{enumerate}

Note that the second condition is the difference between both notions.

In this section we shall consider the orbit
$$
\uu^{c,J}_{L_0}=\{u(L_0): u\in U_c(\h_J)\}.
$$
Note that all orbits are diffeomorphic: if $L\in \Lambda_J(\h)$, there exists $w\in U(\h_J)$ such that $w(L_0)=L$ and therefore
\begin{equation}\label{sondifeomorfas}
{\cal U}_L^{c,J}=\{uw(L_0): u\in U_c(\h_J)\}=w (\uuc),
\end{equation}
because $w^*uw-I=w^*(u-I)w$ is compact.

By the above remarks, clearly one has
$$
\uu^{c,J}_{L_0}=\Lambda_J(\h)\cap G_{res}^0(L_0).
$$
As before, we shall present subspaces $S$ alternatively as projections $p_S$ or symmetries $\e_S$.
If $A\subset\b(\h)$, we shall denote by $A^J$ the set of operators in $A$ which anti-commute with $J$.
Note first that 
$$
\uu^{c,J}_{L_0}\subset \e_{L_0}\  +  \ C(\h)^J.
$$
Indeed, if $u\in U_c(\h_J)$, then 
$$
u\e_{L_0} u^*=\e_{L_0}+(u-I)\e_{L_0}+\e_{L_0}(u^*-I)+(u-I)\e_{L_0}(u^*-I).
$$
Clearly $(u-I)\e_{L_0}+\e_{L_0}(u^*-I)+(u-I)\e_{L_0}(u^*-I)$ is compact, and our claim follows because $u, u^*$ commute with $J$ and $\e_{L_0}$ anti-commutes with $J$.

We shall need the following result, which is a straightforward consequence of the inverse function theorem in Banach spaces. A proof can be found in the appendix of \cite{rae}.

\begin{lem}
Let $G$ be a Banach-Lie group acting smoothly on a Banach space $X$. For a fixed
$x_0\in X$, denote by $\pi_{x_0}:G\to X$ the smooth map $\pi_{x_0}(g)=g\cdot
x_0$. Suppose that
\begin{enumerate}
\item
$\pi_{x_0}$ is an open mapping, when regarded as a map from $G$ onto the orbit
$\{g\cdot x_0: g\in G\}$ of $x_0$ (with the relative topology of $X$).
\item
The differential $d(\pi_{x_0})_1:(TG)_1\to X$ splits: its kernel and range are
closed complemented subspaces.
\end{enumerate}
Then the orbit $\{g\cdot x_0: g\in G\}$ is a smooth submanifold of  $X$, and the
map
$\pi_{x_0}:G\to \{g\cdot x_0: g\in G\}$ is a smooth submersion.
\end{lem}

\begin{prop}\label{subvariedadcompactos}
The set $\uu^{c,J}_{L_0}$ is a complemented differentiable submanifold of $\e_{L_0}+{\cal C}(\h)^J$.  The map
$$
\pi_{L_0}:U_c(\h_J)\to \uu^{c,J}_{L_0} \ , \ \ \pi_{L_0}(u)=u(L_0), \;(\hbox{or equivalently } \pi_{L_0}(u)=u\e_{L_0}u^*)
$$ is a $C^\infty$ submersion.
\end{prop}
\begin{proof}
We use the lemma above. The map $\pi_{L_0}$, regarded as a map from $U_c(\h_J)$ to $\e_{L_0}+{\cal C}(\h)^J$ is clearly $C^\infty$ (in fact $\e_{L_0}+{\cal C}(\h)^J$ is an affine Banach space).  Its differential at $I$ is given by
$$
\delta_{L_0}=d(\pi_{L_0})_I:{\cal C}(\h_J)_{ah}\to {\cal C}(\h)^J \ , \ \ \delta_{L_0}(x)=x\e_{L_0}-\e_{L_0}x.
$$
Here ${\cal C}(\h_J)_{ah}$ denotes the space of anti-hermitic operators in $\h_J$.
The kernel of $\delta_0$ consists of all complex compact anti-hermitic operators which commute with $\e_{L_0}$. An operator commutes with $\e_{L_0}$ if and only if it commutes with the projection $p_{L_0}$. Therefore it is a diagonal matrix in terms of this projection. Thus $\ker \delta_0$ is complemented, a supplement is furnished by the set of all anti-hermitic complex  operators which have co-diagonal matrix with respect to $p_{L_0}$ (or equivalently, anti-commute with $\e_{L_0}$). 

We claim that the range of $\delta_0$ consists of all elements  in ${\cal C}(\h)^J$ which are self-adjoint and anticommute with $\e_{L_0}$. This space is clearly complemented in the (complemented) space ${\cal C}(\h)^J_s$ of self-adjoint elements of ${\cal C}(\h)^J$. Let us prove our claim. If $y=\delta_0(x)=x\e_{L_0}-\e_{L_0}x$ for some $x\in{\cal C}(\h_J)_{ah}$, then it is clearly self-adjoint, compact and  anti-commutes with $J$. Let us prove that it anti-commutes  with $\e_{L_0}$: 
$$
\e_{L_0}y=\e_{L_0}x\e_{L_0}-x \ \hbox{ and } y\e_{L_0}=x- \e_{L_0}x\e_{L_0}.
$$
Conversely, suppose that $y\in {\cal C}(\h)^J_s$ anti-commutes with $\e_{L_0}$, $x=\frac12 y\e_{L_0}$. Then 
clearly $x$ is compact. It is complex, the product of two elements which anti-commute with $J$, commutes with $J$. It is anti-hermitic: $x^*=\e_{L_0}y=-y\e_{L_0}=-x$.  It anti-commutes with $\e_{L_0}$, being the product of $\e_{L_0}$ and $y$ which anti-commutes with $\e_{L_0}$. Finally, 
$$
\delta_0(x)=x\e_{L_0}-\e_{L_0}x=\frac12(y-\e_{L_0}y\e_{L_0})=y.
$$
It remains to prove that $\pi_{L_0}:U_c(\h_J)\to \uu^{c,J}_{L_0}$ is open. In order to prove this, we shall show that it has local continuous cross sections. It suffices to construct a local cross section on a neighborhood of $\e_{L_0}$.
This construction is adapted from \cite{cpr}. For a Lagrangian subspace $L$, consider the element
$$
g=\frac12(I+\e_L\e_{L_0}).
$$
It is invertible if $\e_L$ is close to $\e_{L_0}$ (in fact, it can be shown that it is invertible if $\|\e_L-\e_{L_0}\|<2$). If $L\in\uu^{c,J}_{L_0}\subset \e_{L_0}+{\cal C}(\h)^J$, then
$$
\e_L\e_{L_0}\in(\e_{L_0}+{\cal C}(\h)^J)\e_{L_0}=I+{\cal C}(\h)^J\e_{L_0}=I+{\cal C}(\h_J),
$$
where the last equality follows from the fact that the product of two operators which anti-commute with $J$, commutes with $J$. Thus $g$ is complex and invertible in a neighborhood of $\e_{L_0}$. Note that 
$$
g\e_{L_0}=\frac12(\e_{L_0}+\e_L)=\e_L g.
$$
Note also that $g^*g$ commutes with $\e_{L_0}$. It follows that $u_L=g(g^*g)^{-1/2}$, which is the unitary part in the polar decomposition of $g$, is a continuous map (in the parameter $L$) of complex unitary operators, which conjugates $\e_{L_0}$ and $\e_L$: $\e_L=u_L\e_{L_0}u_L^*$. Let us prove that it takes values in $U_c(\h_J)$. The polar decomposition of $g$ is performed in the $C^*$-algebra $\mathbb{C} I +{\cal C}(\h_J)$. Thus $u_L=\beta I+k$ with $k$ compact. Note that it must be $\beta=1$: indeed, since $g\in I+{\cal C}(\h_J)$, then also $g^*g\in I+{\cal C}(\h_J)$. Thus the spectrum of $g^*g$ only accumulates (eventually) at $1$. It follows that the same is true for $(g^*g)^{-1/2}$, and therefore $(g^*g)^{-1/2}\in I+{\cal C}(\h_J)$. Thus $u_L=g(g^*g)^{-1/2}\in I +{\cal C}(\h_J)$.
\end{proof}
Note that in particular, the tangent spaces of $\uu^{c,J}_{L_0}$ are:
\begin{equation}\label{tangentecompacto}
(T\uuc)_L=\{x\e_L-\e_L x: x\in {\cal C}(\h_J)_{ah}\}=(T\Lambda_J(\h))_L\cap {\cal C}(\h).
\end{equation}
The first equality is a consequence of the fact that $\pi_L$ is a submersion for any $L\in\uuc$. It is clear also that $(T\uuc)_L\subset (T\Lambda_J(\h))_L$ and that $x\e_L-\e_Lx$ is compact. Conversely, suppose that a tangent vector $y\in(T\in\Lambda_J(\h))_L$ is compact. Then, as in the proof of the above theorem, $y=x\e_L-\e_L x$ for $x=\frac12 y \e_L$, which is compact.
\begin{rem}
The space $\uu^{c,J}_{L_0}$ is a complemented submanifold of  $G_{res}^0(L_0)$ and a non-complemented submanifold of $\Lambda_J(\h)$.
Indeed, one must check that the inclusion $i_1:\uuc\hookrightarrow G_{res}(L_0)$ is a splitting immersion, and that $i_1:\uuc\hookrightarrow \Lambda_J(\h)$ is a non-splitting immersion. Note that
$$
(T\uuc)_{L}=\{y\in (TG_{res}^0(L_0))_{L}: yJ=-Jy\},
$$
which is complemented in $(TG_{res}^0(L_0))_{L}$, by the space of complex operators (i.e. operators that commute with $J$) in $(TG_{res}^0(L_0))_{L}$.
On the other hand, it is apparent that
$$
(T\uuc)_{L}=\{x\in (T\Lambda_J)_L: x \hbox{ is compact}\}
$$
is closed but not complemented in $(T\Lambda_J)_L$. To prove this, we may suppose $\h=\k\times \k$ and $L=L_0=G_I$ (recall that all orbits are diffeomorphic, the diffeomorphism is implemented by a linear complex unitary operator as in eq. (\ref{sondifeomorfas})). Then, by the computations in Section 3,
\begin{eqnarray}
(T\uuc)_{L_0} &=&\left\{\left( \begin{array}{cc} 0 & y \\ -y & 0 \end{array}\right):  y^*=-y, y \hbox{ compact }\right\}\nonumber\\
&\subset &(T\Lambda_J)_{L_0}=\left\{\left( \begin{array}{cc} 0 & y \\ -y & 0 \end{array}\right) :  y^*=-y\right\},\nonumber
\end{eqnarray}
which apparently is a closed but non split inclusion, and it is easy to see that $\uuc\subset \Lambda_J(\h)$ has the subspace topology since if $u=e^z$ is close to $1$, then $z=\log(u)=\sum \frac{(-1)^{n+1}}{n+1}(u-1)^{n+1}$ is a compact operator.
\end{rem}

The linear connection of $\Lambda_J(\h)$ restricts to $\uu^{c,J}_{L_0}$.
\begin{prop}\label{conexioncompactos}
Let $X$ and $Y$ be two tangent vector fields in $\uu^{c,J}_{L_0}$. Denote by $\nabla$ the linear connection in $\Lambda_J(\h)$. Then 
$$
\nabla_XY\in T \uu^{c,J}_{L_0}.
$$
\end{prop}
\begin{proof}
Let $X(t)$ be a tangent field along $\e(t)$ in $\uuc$. We must show that the covariant derivative in $T\Lambda_J(\h)$,
$$
\frac{D}{dt}X=\Pi_\e(\dot{X}),
$$
takes values in $T\uuc$. By  formula (\ref{tangentecompacto}), it suffices to show that it takes compact values. Recall that $\Pi_\e(a)=\frac12(a-\e a\e)$. Then the proof follows, since the $\dot{X}(t)$ is compact because $X(t)$ is compact for all $t$.
\end{proof}
The argument of Theorem \ref{hopfrinowlagrange} adapts to this submanifold.
\begin{thm}\label{minimalidadcompacta}
Let $S_0,S_1\in  \uu^{c,J}_{L_0}$. Then there exists a geodesic of the linear connection of $\uu^{c,J}_{L_0}$ which joins them. The geodesic has minimal length (with respect to the Finsler metric induced by the operator norm). If $\|\e_{S_0}-\e_{S_1}\|<2$, the geodesic is unique. In any case, $\uu^{c,J}_{L_0}$ is geodesically convex, in the sense that if $S_0,S_1\in \uu^{c,J}_{L_0}$ and $\gamma$ is a geodesic of the connection joining them, then $\gamma\subset \uu^{c,J}_{L_0}$.
\end{thm}
\begin{proof}
It suffices to prove the result for $S_0=L_0$. We proceed as in Th. \ref{hopfrinowlagrange}, we must show that the complex anti-hermitic operator $z$, which anti-commutes with $\e_{L_0}$,   and verifies that $e^z(L_0)=S_1$, or equivalently $e^z\e_{L_0}e^{-z}=e^{2z} \e_{L_0}$, is compact.  This operator operator $z$ is constructed as the sum of $z_0$ acting in the generic part of the subspaces $L_0$ and $S_1$, plus $z_2$, which is $\frac{\pi}{2} J$ acting in $(L_0\cap S_1^\perp)\oplus (L_0^\perp \cap S_1)$.

First note that $(L_0\cap S_1^\perp)\oplus (L_0^\perp \cap S_1)$ is finite dimensional. This follows from the fact that $S_1$ belongs to the restricted Grassmannian $G_{res}^0(L_0)$: 
$$
p_{L_0}|_{S_1}:S_1\to L_0
$$
is a Fredholm operator, thus $S_0^\perp\cap L_0=\ker (p_{L_0}|_{S_1})$ is finite dimensional. Since $J(S_0^\perp\cap L_0)=S_0\cap L_0^\perp$, our claim is proven.

Consider now the generic part. Again using Halmos' trick, it suffices to consider the case when the projections onto the subspaces $L_0$ and $S_1$ are given (respectively) by the matrices
$$
p_0= \left( \begin{array}{cc} 1 & 0 \\ 0 & 0 \end{array} \right) \ \ \hbox{ and } \ \ 
p_1=\left( \begin{array}{cc} c^2 & cs \\ cs & s^2 \end{array} \right), 
$$
acting in $\k\times\k$, with $c,s$ commuting positive operators in $\k$, such that $c^2+s^2=I$. Therefore there exists $x\ge 0$, $\|x\|\le \pi/2$ in $\k$ such that $c=\cos(x)$ and $s=\sin(x)$. Since $p_1$ lies in the restricted Grassmannian of $p_0$, it follows that $p_1|_{\ker p_0}$ is compact. That is, $\cos(x)\sin(x)+\sin(x)^2$ is compact in $\k$. Note that the entire function $f(t)=\cos(t)\sin(t)-\sin(t)^2$ is of the form $f(t)=tg(t)$ for $g$ an entire function, non vanishing in the interval $[0,\pi/2]$. It follows that $x=f(x)\frac{1}{g}(x)$ is compact. 
This completes the proof of the first statement.

For the last statement, we assume again that $S_0=L_0$, and we regard $S_1\in \uu^{c,J}_{L_0}$ as the projection $p_{S_1}$ given by
$$
p_{S_1}=\left( \begin{array}{cc}  \sin(2y)+I & \cos (2y) \\ \cos(2y) & I-\sin(2y) \end{array} \right),
$$
with  $y$ self-adjoint and compact in $\k$. Assume that $\|y\|=\frac{\pi}{2}$ (the case $\|y\|<\frac{\pi}{2}$ has been proved). Let $y'$ be any other self-adjoint operator such that $\|y'\|=\frac{\pi}{2}$ and
\begin{equation}\label{iguales}
\left( \begin{array}{cc}  \sin(2y')+I & \cos (2y') \\ \cos(2y') & I-\sin(2y') \end{array} \right)=\left( \begin{array}{cc}  \sin(2y)+I & \cos (2y) \\ \cos(2y) & I-\sin(2y) \end{array} \right),
\end{equation}
i.e. $y'$ is the initial speed of a geodesic $\gamma$ starting at $S_0$ and ending at $S_1$.  Then 
$$
y'=\frac{\pi}{2}p_{+\infty}-\frac{\pi}{2}p_{-\infty}+\sum\limits_{|\mu_n|<\frac{\pi}{2}} \mu_n p_n,
$$
where $p_{+\infty,},p_{-\infty},p_n$ are mutually orthogonal projections. Indeed, from equation (\ref{iguales}), we obtain $\sin(2y)=\sin(2y')$ and $\cos(2y)=\cos(2y')$. Equivalently, $e^{2iy}=e^{2iy'}$, and in particular $e^{2i\sigma(y)}=e^{2i\sigma(y')}$. From this and the fact that $\|y\|=\|y'\|=\frac{\pi}{2}$, it follows that the $p_n$ are finite rank projections and moreover $\mu_n\to 0$. Since $y$ is a compact operator, the multiplicity of $-1$ in the spectrum of $e^{2iy'}$ is finite, and then it must be that $p_{+\infty}$, $p_{-\infty}$ are also of finite rank by the stated equality $e^{2iy}=e^{2iy'}$. This proves that $y'$ is also compact, and then $\gamma\subset\uu^{c,J}_{L_0}$.
\end{proof}
Note that if the eigenspace of $e^{2iy}$  corresponding to $-1$ is one dimensional, there are two minimal geodesics joining $G_I$ and $e^{iy}(G_I)$. Otherwise, for dimension greater or equal than $2$, there are infinitely many. Thus in $\uuc$, two points are joined by one, two or infinitely many minimal geodesics.

One may replace the Fredholm group by any of the Schatten unitary groups $U_k(\h_J)$ for $k\ge 2$. Namely let $B_k(\h_J)$ denote $k$-Schatthen ideal
$$
B_k(\h)=\{a\in B(\h): Tr(|a|^k)<\infty\},
$$
with the $k$-norm $\|a\|_k=Tr(|a|^k)^{1/k}$. Denote by
$$
U_k(\h_J)=\{u\in U(\h_J): u-I\in B_k(\h_J)\}.
$$
Consider the orbit of  $L_0\in\Lambda_J(\h)$ under the action of $U_k$,
$$
\uuk=\{u(L_0): u\in U_k(\h_J)\}.
$$
The $k$-Schatten restricted Grassmannian is defined accordingly. A closed subspace $S\subset \h$ belongs to the $k$-Schatten restricted Grassmannian $G^{0,k}_{res}(L_0)$ if
\begin{enumerate}
\item
$p_{L_0}|_{S}:S\to L_0 \hbox{ is invertible modulo } B_k(\h) \hbox{ and has index } 0,$
and 
\item
$p_{L_0}|_{T}:T\to L_0 \hbox{ belongs to } B_k(T,L_0).$
\end{enumerate}
The group $U_k(\h)$ acts transitively on $G^{0,k}_{res}(L_0)$. Therefore $\uuk=\Lambda_J(\h)\cap G^{0,k}_{res}(L_0)$.
Also a straightforward computation shows that, regarding subspaces as symmetries
$$
\uuk\subset \e_{L_0}+B_k(\h)^J.
$$
Therefore one may use the $k$ norm to measure the distance between elements in $\uuk$. We may thus ask the analogous questions for this norm. 
\begin{itemize}
\item First, if $\uuk$ is a submanifold of the $k$-Schatten affine space $\e_{L_0}+B_k(\h)^J$. 
\item Second, if the connection in $\Lambda_J(\h)$ restricts to this manifold as well as in the compact case. 
\item Third, if so, and if we endow the tangent spaces of $\uuk$ with the $k$-norm, examine the minimality properties of the geodesics of the connection.
\end{itemize}

The first question is answered affirmatively. The proof is similar to the one given for Proposition \ref{subvariedadcompactos}, replacing $C(\h)$ with $B_k(\h)$. Let us sketch it here.
\begin{prop}
The set $\uuk$ is a differentiable submanifold of $\e_{L_0}+B_k(\h)^J$.  The map
$$
\pi_{L_0}:U_k(\h_J)\to \uuk, \ \ \pi_{L_0}(u)=u(L_0), \;(\hbox{or equivalently } \  \pi_{L_0}(u)=u\e_{L_0}u^*)
$$ is a $C^\infty$ submersion.
\end{prop}
\begin{proof}
The first part of the argument follows verbatim as in Prop. \ref{subvariedadcompactos}. The map
$$
\delta_{L_0}=d(\pi_{L_0})_I:B_k(\h_J)_{ah}\to B_k(\h)^J \ , \ \ \delta_{L_0}(x)=x\e_{L_0}-\e_{L_0}x.
$$
is proven to have complemented kernel and range in the same fashion as above.

In order to prove that $\pi_{L_0}:U_c(\h_J)\to \uu^{c,J}_{L_0}$ is open, we show that it has local continuous cross sections. Let us prove that the cross section defined in Prop. \ref{subvariedadcompactos} on a neighborhood of $\e_{L_0}$ adapts to this situation.
For a Lagrangian subspace $L$,  the element $g=\frac12(I+\e_L\e_{L_0})$ is invertible if $\e_L$ is close to $\e_{L_0}$. If $L\in\uuk\subset \e_{L_0}+B(\h)^J$, then
$$
\e_L\e_{L_0}\in(\e_{L_0}+{\cal C}(\h)^J)\e_{L_0}=I+B_k(\h)^J\e_{L_0}=I+B_k(\h_J).
$$
The unitary part $u_L$ in the polar decomposition of $g$, $u_L=g(g^*g)^{-1/2}$, is a continuous map (in the parameter $L$) of complex unitary operators, in the topology given by the $k$ norm. Note that $B_k(\h_J)$ is a *-Banach algebra, and that the computation done to obtain $u_L$ is performed in the unitization of this algebra. The operations involved (product, involution, inversion, square root) are continuous in the unitization.

It conjugates $\e_{L_0}$ and $\e_L$. Let us prove that it takes values in $U_k(\h_J)$.  By Prop. \ref{subvariedadcompactos}, it takes values in $U_c(\h_J)$. On the other hand, as remarked above, it also takes values in $\mathbb{C}I+B_k(\h_J)$, the unitization of $B_k(\h_J)$.
\end{proof}

The second question, that the connection of $\Lambda_J(\h)$ restrict well to $\uuk$, also follows almost verbatim form the analogous fact for the compact case (Prop. \ref{conexioncompactos}). If $X$ and $Y$ are  two tangent vector fields in $\uuk$, and  $\nabla$ the linear connection in $\Lambda_J(\h)$, then 
$$
\nabla_XY\in T \uuk.
$$
Indeed, if $X(t)$ is a tangent field along $\e(t)$ in $\uuk$, its covariant derivative $\frac{D}{dt}X=\Pi_\e(\dot{X}),
$ takes values in $T\uuk$, because $X(t)\in B_k(\h)$  for all $t$ and therefore $\dot{X}(t)\in B_k(\h)$. 

\medskip

Let us finish this paper with a proof of the minimality results on geodesics of $\uuk$ in the $k$-norm. 

\begin{thm}
Let $S_0,S_1\in  \uuk$. Then there exists a geodesic of the linear connection of $\uuk$ which joins them. The geodesic has minimal length (with respect to the Finsler metric induced by the $k$-norm). If $\|\e_{S_0}-\e_{S_1}\|<2$, the geodesic is unique. The manifold $\uuk$ is geodesically convex: if $S_0,S_1\in \uuk$ and $\gamma$ is a geodesic of the connection joining them, then $\gamma\subset \uuk$.
\end{thm}
\begin{proof}
As in the proof of Th. \ref{minimalidadcompacta}, $(S_0\cap {\cal S}_1^\perp)\oplus(S_0^\perp\cap S_1)$ is finite dimensional. Thus it suffices to regard the generic part of $S_0$ and $S_1$. Note that the argument in Th. \ref{minimalidadcompacta} used to prove that $x$ (such that $\cos(x)=c$ and $\sin(x)=s$) is compact, shows in fact that in this case $x\in B_k(\k)$.
The argument proving the geodesic convexity can also be adapted to this case, because $y$ and $y'$ share the same eigenvalues $\mu_n$ (with $|\mu_n|<\pi/2$).

It remains to be proved that the geodesics are minimal when measured with the $k$-norm. To prove this we use the fact that in the group $U_k(\h_J)$, the curves of the form $\delta(t)=u e^{tx}$ with $x^*=-x\in B_k(\h_J)$ are minimal for $t\in[0,1]$ provided that the operator norm $\|x\|\le \pi$ \cite{upe}. Let $S_0,S_1\in\uuk$ and let $\e(t)=e^{2tz}\e_{S_0}$ be the geodesic joining them ($t\in[0,1]$). Note that $z^*=-z\in B_k(\h_J)$ and $\|z\|\le \pi/2$. Let $\gamma$ be another curve in $\uuk$ with the same endpoints. Since both $\e$ and $\gamma$ lie in particular in $\e_{S_0}+B_k(\h)$, it follows that $\gamma \e_{S_0}$ and $\e \e_{S_0}=e^{2tz}$ are curves in $I+B_k(\h_J)$, i.e. they are curves in $U_k(\h_J)$, joining the same endpoints, since $\|2z\|\le \pi$. By the fact that multiplying by a fixed element $\e_{S_0}$ is an isometric map (between $\uuk$ and $U_k(\h_J)$), and the minimality result in $U_k(\h_J)$, it follows that if $L_k$ denotes the length of a curve in the $k$-norm (either in $U_k(\h_J)$ or $\uuk$), then
$$
L_k(\e)=L_k(\e \e_{S_0})\le L_k(\gamma \e_{S_0})=L_k(\gamma).
$$
\end{proof}

\bigskip

\noindent
Esteban Andruchow and Gabriel Larotonda\\
Instituto de Ciencias \\
Universidad Nacional de Gral. Sarmiento \\
J. M. Gutierrez 1150 \\
(1613) Los Polvorines \\
Argentina  \\
e-mails: eandruch@ungs.edu.ar, glaroton@ungs.edu.ar

\end{document}